\documentclass[A4, 11pt]{article}
\usepackage{amscd, amssymb, amsmath, amsthm}
\usepackage{latexsym}
\usepackage{upref}
\usepackage{epsfig}
\usepackage{mathrsfs}
\usepackage{float, floatflt}
\usepackage{indentfirst}
\usepackage{hyperref}
\hypersetup{
    colorlinks=true,
    linkcolor=blue,
    citecolor=blue,
    filecolor=blue,
    urlcolor=blue,
}
\usepackage{graphics,graphicx,color}
\usepackage{cite}
\usepackage{enumitem}
\usepackage{fancyhdr} 
\usepackage[T1]{fontenc}
\usepackage{mathptmx} 

\AtBeginDocument{
  \DeclareSymbolFont{AMSb}{U}{msb}{m}{n}
  \DeclareSymbolFontAlphabet{\mathbb}{AMSb}}  

\usepackage{rotating} 
\usepackage{booktabs} 
\usepackage{pb-diagram,pb-xy}       
\usepackage{tikz}                   
\usetikzlibrary{matrix,arrows}
\usepackage{multirow} 
\usepackage{color, colortbl} 
\definecolor{Gray}{gray}{0.93} 
\definecolor{LightCyan}{rgb}{0.88,1,1}

\theoremstyle{plain}
\newtheorem{theorem}{Theorem}[section]
\newtheorem{lemma}[theorem]{Lemma}

\newtheorem{corollary}[theorem]{Corollary}
\theoremstyle{definition}

\theoremstyle{remark}

\numberwithin{equation}{section}

\makeatletter
\def\th@plain{%
  \thm@notefont{}
  \itshape 
}
\def\th@definition{%
  \thm@notefont{}
  \normalfont 
} \makeatother

\setlist{font=\normalfont}

\DeclareMathAlphabet{\cols}{OMS}{cmsy}{m}{n} %

\newcommand{\R}{\mathbb{R}}
\newcommand{\B}{\mathbb{B}}
\newcommand{\set}[1]{\{#1\}}
\newcommand{\cset}[2]{\set{{#1}\colon{#2}}}
\newcommand{\norm}[1]{\|#1\|}
\newcommand{\Bp}[1]{\left(#1\right)}
\newcommand{\gen}[1]{\langle#1\rangle}
\DeclareSymbolFont{newfont}{OML}{cmm}{m}{it} 
\DeclareMathSymbol{\vrho}{3}{newfont}{37}
\newcommand{\gyr}[2]{{\mathrm{gyr}[{#1}]}{#2}}
\newcommand{\aut}[1]{\mathrm{Aut}\,{(#1)}}
\newcommand{\igyr}[2]{{\mathrm{gyr^{-1}}[{#1}]}{#2}}

\newcommand{\Cl}[1]{\mathrm{C}\ell_{#1}}
\newcommand{\lsum}[2]{\displaystyle\sum_{#1}^{#2}}
\newcommand{\mul}[1]{{#1}^\times}
\newcommand{\abs}[1]{|#1|}

\newcommand{\Or}[1]{\mathrm{O}\,({#1})}
\newcommand{\res}[2]{{#1}\hskip-3pt\mid_{#2}}
\newcommand{\Gyr}[2]{{\mathrm{Gyr}[{#1}]}{#2}}
\newcommand{\Iso}[1]{\mathrm{Iso}\,{(#1)}}

\newcommand{\qt}[1]{``#1''}

\renewcommand{\vec}[1]{\mathbf{#1}}

\begin{document}
\title{\bf{An inequality related to M\"{o}bius transformations}}
\author{
Themistocles M. Rassias\\
Department of Mathematics\\
National Technical University of Athens\\
Zografou, Campus\\
15780 Athens, Greece\\
{\tt trassias@math.ntua.gr}\\[0.3cm]
Teerapong Suksumran\footnote{Corresponding author.}\,\,\,\href{https://orcid.org/0000-0002-1239-5586}{\includegraphics[scale=1]{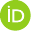}}\\
Department of Mathematics\\
Faculty of Science, Chiang Mai University\\
Chiang Mai 50200, Thailand\\
{\tt teerapong.suksumran@cmu.ac.th}}
\date{}
\maketitle


\begin{abstract}
The open unit ball $\mathbb{B} = \{\vec{v}\in\mathbb{R}^n\colon\|\vec{v}\|<1\}$ is endowed with M\"{o}bius \mbox{addition} $\oplus_M$ defined by
$$
\vec{u}\oplus_M\vec{v} = \dfrac{(1 + 2\langle\vec{u},\vec{v}\rangle +
\|\vec{v}\|^2)\vec{u} + (1 - \norm{\vec{u}}^2)\vec{v}}{1 +
2\gen{\vec{u},\vec{v}} + \|\vec{u}\|^2\|\vec{v}\|^2}
$$
for all $\vec{u},\vec{v}\in\B$. In this article, we prove the inequality
$$
\dfrac{\|\vec{u}\|-\|\vec{v}\|}{1+\|\vec{u}\|\|\vec{v}\|}\leq \|\vec{u}\oplus_M \vec{v}\| \leq \dfrac{\|\vec{u}\|+\|\vec{v}\|}{1-\|\vec{u}\|\|\vec{v}\|}
$$
in $\mathbb{B}$. This leads to a new metric on $\B$ defined by $$d_T(\vec{u},\vec{v}) = \tan^{-1}{\|-\vec{u}\oplus_M\vec{v}\|},$$ which turns out to be an invariant of M\"{o}bius transformations on $\R^n$ carrying $\B$ onto itself. We also compute the isometry group of $(\B, d_T)$ and give a parametrization of the isometry group by vectors and rotations.
\end{abstract}
\thispagestyle{empty}

\newpage

\section{The unit ball of $\boldsymbol{n}$-dimensional Euclidean space $\boldsymbol{\R^n}$}
Let $\B$ denote the open unit ball of $n$-dimensional Euclidean space $\R^n$, that is,
\begin{equation}
\B =\cset{\vec{v}\in\R^n}{\norm{\vec{v}}<1},
\end{equation}
where $\norm{\cdot}$ denotes the usual Euclidean norm on $\R^n$.  It is known in the literature that $\B$ forms a bounded symmetric domain, naturally associated with the
Poincar\'{e} and Beltrami--Klein models of $n$-dimensional hyperbolic geometry. In fact, the Poincar\'{e} metric $d_P$ corresponding to a curvature of $-1$ is given by
\begin{equation}
d_P(\vec{x}, \vec{y}) = \cosh^{-1}\Bp{1+\dfrac{2\norm{\vec{x}-\vec{y}}^2}{(1-\norm{\vec{x}}^2)(1-\norm{\vec{y}}^2)}}
\end{equation}
for all $\vec{x}, \vec{y}\in \B$ \cite[p. 1232]{SKJL2013UBL}. Further, the Cayley--Klein metric associated with the Beltrami--Klein model is defined via cross-ratios; see, for instance, \cite[p. 1233]{SKJL2013UBL}.

From an algebraic point of view, the unit ball has a group-like structure when it is endowed with {\it M\"{o}bius addition} $\oplus_M$ defined by
\begin{equation}\label{eqn: Euclidean Mobius addition}
\vec{u}\oplus_M\vec{v} = \dfrac{(1 + 2\gen{\vec{u},\vec{v}} +
\norm{\vec{v}}^2)\vec{u} + (1 - \norm{\vec{u}}^2)\vec{v}}{1 +
2\gen{\vec{u},\vec{v}} + \norm{\vec{u}}^2\norm{\vec{v}}^2}.
\end{equation}
M\"{o}bius addition governs the unit ball in the same way that ordinary vector addition governs the Euclidean space; see, for instance, \cite{AU2008FMG, YFTS2005PAH, JR2006FHM}. Further, M\"{o}bius addition induces the well-known M\"{o}bius transformation of $\B$ of the form
\begin{equation}
L_{\vec{u}}(\vec{v}) = \vec{u}\oplus_M\vec{v} = \dfrac{(1 + 2\gen{\vec{u},\vec{v}} +
\norm{\vec{v}}^2)\vec{u} + (1 - \norm{\vec{u}}^2)\vec{v}}{1 +
2\gen{\vec{u},\vec{v}} + \norm{\vec{u}}^2\norm{\vec{v}}^2},
\end{equation}
called the {\it hyperbolic translation} by $\vec{u}$, for all $\vec{u}\in\B$ \cite[p.~124]{JR2006FHM}. A remarkable result of Kim and Lawson shows strong connections between the geometric and algebraic structures of the unit ball. In fact, they relate the Poincar\'{e} metric with M\"{o}bius addition:
\begin{equation}\label{eqn: Poincare metric and Mobius addition}
d_P(\vec{x},\vec{y}) = 2\tanh^{-1}{\norm{-\vec{x}\oplus_M\vec{y}}}
\end{equation}
for all $\vec{x},\vec{y}\in\B$; see Theorem 3.7 of \cite{SKJL2013UBL}. Equation \eqref{eqn: Poincare metric and Mobius addition}  includes what Ungar refers to as a {\it gyrometric} \cite[Definition 6.8]{AU2008AHG}. More precisely, the (M\"{o}bius) gyrometric and the rapidity metric of $(\B, \oplus_M)$ are defined by
\begin{equation}\label{eqn: Mobius gyrometric}
\varrho_M(\vec{x}, \vec{y}) = \norm{-\vec{x}\oplus_M\vec{y}}
\end{equation}
and by \begin{equation}
d_M(\vec{x},\vec{y}) = \tanh^{-1}{(\varrho_M(\vec{x}, \vec{y}))}
\end{equation}
for all $\vec{x},\vec{y}\in\B$, respectively.

\subsection{A nonassociative structure of the unit ball}
The space $(\B, \oplus_M)$ shares many properties with abelian groups, called by some a {\it gyrocommutative gyrogroup} and by others a {\it Bruck loop} or a {\it K-loop}. Henceforth,  $(\B, \oplus_M)$ is referred to as the {\it M\"{o}bius gyrogroup}. 

The group-like axioms satisfied by the M\"{o}bius gyrogroup are as follows.
\begin{enumerate}[label=(\Roman*)]
    \item ({\sc identity}) The zero vector $\vec{0}$ satisfies $\vec{0}\oplus_M \vec{v} = \vec{v} = \vec{v}\oplus_M\vec{0}$ for all $\vec{v}\in \B$.
    \item ({\sc inverse}) For each $\vec{v}\in \B$, the negative vector $-\vec{v}$ belongs to $\B$ and satisfies $$(-\vec{v})\oplus_M \vec{v} = \vec{0} = \vec{v}\oplus_M(-\vec{v}).$$
    \item\label{item: gyroassociative law} ({\sc the gyroassociative law}) For all $\vec{u}, \vec{v}\in \B$, there are automorphisms $\gyr{\vec{u},\vec{v}}{}$ and $\gyr{\vec{v},\vec{u}}{}$ in $\aut{\B, \oplus_M}$ such that
\begin{equation*}
\vec{u}\oplus_M (\vec{v}\oplus_M \vec{w}) = (\vec{u}\oplus_M \vec{v})\oplus_M\gyr{\vec{u}, \vec{v}}{\vec{w}}
\end{equation*}
and 
\begin{equation*}
(\vec{u}\oplus_M \vec{v})\oplus_M \vec{w} = \vec{u}\oplus_M (\vec{v}\oplus_M\gyr{\vec{v}, \vec{u}}{\vec{w}})
\end{equation*}
for all $\vec{w}\in \B$.
    \item ({\sc the loop property}) For all $\vec{u},\vec{v}\in \B$, 
$$
\gyr{\vec{u}\oplus_M \vec{v}, \vec{v}}{} = \gyr{\vec{u}, \vec{v}}{}\quad\textrm{and}\quad \gyr{\vec{u}, \vec{v}\oplus_M \vec{u}}{} = \gyr{\vec{u}, \vec{v}}{}.
$$    
\item ({\sc the gyrocommutative law}) For all $\vec{u},\vec{v}\in \B$, 
$$
\vec{u}\oplus_M\vec{v} = \gyr{\vec{u}, \vec{v}}{(\vec{v}\oplus_M\vec{u})}.
$$
\end{enumerate}
The automorphism $\gyr{\vec{u}, \vec{v}}{}$ mentioned in Item \ref{item: gyroassociative law} is called the {\it gyroautomorphism} generated by $\vec{u}$ and $\vec{v}$. It is uniquely determined by its generators via the {\it gyrator identity} described by the formula
\begin{equation}
\gyr{\vec{u},\vec{v}}{\vec{w}} = -(\vec{u}\oplus_M \vec{v})\oplus_M (\vec{u}\oplus_M (\vec{v}\oplus_M \vec{w}))
\end{equation}
for all $\vec{w}\in\B$. Sometimes it is convenient to denote $-\vec{v}$ by $\ominus\vec{v}$, the (unique) \mbox{inverse} of $\vec{v}$ with respect to M\"{o}bius addition. Some elementary properties of the M\"{o}bius gyrogroup are collected in Table \ref{tab: properties of gyrogroups}.

\begin{table}[ht]
\centering
{\small
\begin{tabular}{ll}\hline
\rowcolor{LightCyan}{\hskip2cm}{\sc gyrogroup identity} & {\hskip0.5cm}{\sc name/reference}\\ \hline
$L_{\ominus \vec{u}} = L_{\vec{u}}^{-1}$ & Inverse of gyrotranslation\\ 
\rowcolor{Gray}$\ominus \vec{u}\oplus_M(\vec{u}\oplus_M \vec{v}) = \vec{v}$ & Left cancellation law\\
$\ominus (\vec{u}\oplus_M \vec{v}) = \gyr{\vec{u}, \vec{v}}{(\ominus \vec{v}\oplus_M \ominus\vec{u})}$ & cf. $(gh)^{-1} = h^{-1}g^{-1}$\\
\rowcolor{Gray} $(\ominus \vec{u}\oplus_M \vec{v})\oplus_M\gyr{\ominus \vec{u}, \vec{v}}{(\ominus \vec{v}\oplus_M \vec{w})} = \ominus \vec{u}\oplus_M \vec{w}$ & cf. $(g^{-1}h)(h^{-1}k) = g^{-1}k$\\
$\gyr{\ominus \vec{u}, \ominus \vec{v}}{} = \gyr{\vec{u}, \vec{v}}{}$ & Even property\\
\rowcolor{Gray}$\gyr{\vec{v}, \vec{u}} \,=\, \igyr{\vec{u}, \vec{v}}{}$, the inverse of $\gyr{\vec{u}, \vec{v}}{}$ & Inversive symmetry\\
\hline
\end{tabular}}\vskip5pt
\caption{Properties of the M\"{o}bius gyrogroup (cf. \cite{AU2008AHG, TS2016TAG}).}\label{tab: properties of gyrogroups}
\end{table}

\subsection{Isometries of the unit ball}
It is known in the literature that the transformation $L_\vec{u}\colon \vec{v}\mapsto \vec{u}\oplus_M \vec{v}$ preserves the gyrometric $\varrho_M$; see, for instance, \cite[Lemma 3.2 (v)]{SKJL2013UBL}. Thus, $L_\vec{u}$ preserves the rapidity metric $d_M$. In fact, every isometry of $(\B, d_M)$ must be of the form $L_\vec{u}\circ \tau$, where $\tau$ is the restriction of an orthogonal transformation on $\R^n$ to the unit ball, due to the fact that any M\"{o}bius transformation that fixes $\vec{0}$ is orthogonal. The following theorem shows that the metric geometry of $\B$ with respect to $d_M$ is  homogeneous.

\begin{theorem}[Homogeneity]
For each pair of points $\vec{x}$ and $\vec{y}$ in $\B$, there is an isometry $T$ of $(\B, d_M)$ such that $T(\vec{x}) = \vec{y}$. In particular, $\B$ is homogeneous.
\end{theorem}
\begin{proof}
Let $\vec{x}, \vec{y}\in\B$. Define $T = L_{\vec{y}}\circ L_{\ominus \vec{x}}$. Then $T$ is an isometry of $\B$, being the composite of isometries of $\B$. Further,
$T(\vec{x}) = \vec{y}\oplus_M(\ominus \vec{x}\oplus_M \vec{x}) = \vec{y}$.
\end{proof}

By using the gyrogroup formalism, a {\it point-reflection} symmetry of $\B$ is easy to construct, as shown in the following theorem.

\begin{theorem}[Symmetry]
For each point $\vec{x}\in\B$, there is a symmetry $S_{\vec{x}}$ of $\B$; that is, $S_{\vec{x}}$ is an isometry of $(\B, d_M)$ such that $S_{\vec{x}}^2$ is the identity transformation $I$ of $\B$ and $\vec{x}$ is the unique fixed point of $S_{\vec{x}}$.
\end{theorem}
\begin{proof}
Let $\iota$ be the inversion map of $\B$, that is, $\iota(\vec{v}) = \ominus \vec{v}$ for all $\vec{v}\in\B$. Since $\ominus \vec{v} = -\vec{v}$ for all $\vec{v}\in\B$, $\iota$ is simply the negative map: $\vec{v}\mapsto -\vec{v}$. Note that $\iota$ is an isometry of $(\B, d_M)$ for $\iota$ is linear and preserves the Euclidean norm. Furthermore, $\iota(\vec{v}) = \vec{v}$ if and only if $\vec{v} = \vec{0}$. 

Given $\vec{x}\in\B$, define $S_{\vec{x}} = L_\vec{x}\circ \iota\circ L_{\ominus \vec{x}}$. Then $S_{\vec{x}} = L_\vec{x}\circ \iota\circ L_{\vec{x}}^{-1}$ and so 
$$
S_\vec{x}^2 = (L_\vec{x}\circ \iota\circ L_{\vec{x}}^{-1})\circ(L_\vec{x}\circ \iota\circ L_{\vec{x}}^{-1}) = L_{\vec{x}}\circ\iota^2\circ L_{\vec{x}}^{-1} = L_{\vec{x}}\circ L_{\vec{x}}^{-1} = I.
$$
Note that $S_\vec{x}\ne I$; otherwise, we would have $L_\vec{x}\circ \iota\circ L_{\vec{x}}^{-1} = I$ and would have $\iota = I$, a contradiction. It is clear that $S_\vec{x}$ is an isometry of $\B$. By construction, $\vec{x}$ is a fixed point of $S_{\vec{x}}$. Suppose that $\vec{y}$ is a fixed point of $S_{\vec{x}}$, that is, $S_{\vec{x}}(\vec{y}) = \vec{y}$. It follows that
$
\vec{x}\oplus_M \iota(\ominus \vec{x}\oplus_M \vec{y}) = \vec{y}
$
and hence $\iota(\ominus \vec{x}\oplus_M \vec{y}) = \ominus \vec{x}\oplus_M \vec{y}$. As mentioned previously, $\vec{0}$ is the unique fixed point of $\iota$ and so $\ominus \vec{x}\oplus_M \vec{y} = \vec{0}$. This implies that $\vec{x} = \vec{y}$.
\end{proof}

We close this section with the following theorem whose proof is straight-forward (and so is omitted).

\begin{theorem}
If $\tau\in\aut{\B, \oplus_M}$ and $\norm{\tau(\vec{v})} = \norm{\vec{v}}$ for all $\vec{v}\in\B$, then $\tau$ is an isometry of $\B$ with respect to $d_M$. In particular, the gyroautomorphisms of $(\B, \oplus_M)$ are isometries.
\end{theorem}

\section{The negative Euclidean space and its Clifford algebra}
It seems that the formalism of Clifford algebras is a suitable tool for the study of the M\"{o}bius gyrogroup \cite{MFGR2011MGC, JL2010CAM}. Let us begin with the definition of an underlying vector space that will be used to built a unital associative algebra in which M\"{o}bius addition has a compact formula. The {\it negative} Euclidean space has $\R^n$ as the underlying vector space, but its inner product is a variant of the  Euclidean inner product defined by
\begin{equation}\label{eqn: inner product of negative space}
B(\vec{u},\vec{v}) = -\gen{\vec{u},\vec{v}},\qquad \vec{u},\vec{v}\in\R^n.
\end{equation}
Note that \eqref{eqn: inner product of negative space} defines a nondegenerate symmetric bilinear form on $\R^n$. Also, the associated quadratic form is given by $Q(\vec{v}) = -\norm{\vec{v}}^2$ for all $\vec{v}\in\R^n$.

The negative Euclidean space induces a real unital associative algebra, which is unique up to isomorphism, called the {\it Clifford algebra} of $(\R^n, B)$ denoted by $\Cl{n}$ \cite{JL2010CAM}. To describe the structure of $\Cl{n}$, let $\set{\vec{e}_1, \vec{e}_2,\ldots, \vec{e}_n}$ be the standard basis of $\R^n$. Then $\Cl{n}$ has a basis of the form
\begin{equation}
\cset{e_I}{I = \emptyset\textrm{ or }I = \set{1\leq i_1 < i_2 < \cdots < i_k\leq n}},
\end{equation}
where $e_I = \vec{e}_{i_1}\vec{e}_{i_2}\cdots \vec{e}_{i_k}$ for $I = \set{1\leq i_1 < i_2 < \cdots < i_k\leq n}$ and $e_{\emptyset} = 1$, the multiplicative identity of $\Cl{n}$. Hence, a typical element of $\Cl{n}$ is of the form $\lsum{I}{}\lambda_Ie_I$ with $\lambda_I$ in $\R$. The binary operations of vector addition and scalar multiplication in $\Cl{n}$ are defined pointwise. The product of two elements in $\Cl{n}$ is obtained by using the distributive law (but not assuming that algebra multiplication is commutative) subject to the defining relations
\begin{equation}
\vec{e}_i^2 = -1\quad\textrm{and}\quad \vec{e}_i\vec{e}_j = -\vec{e}_j\vec{e}_i
\end{equation}
for all $i, j\in\set{1,2,\ldots, n}$ with $i\ne j$. The base field $\R$ is embedded into $\Cl{n}$ by the map $\lambda\mapsto \lambda 1$, and the original space $\R^n$ is embedded into $\Cl{n}$ by the inclusion map \cite[Section 3]{TS2016TAG}.

There is a unique involutive algebra anti-automorphism of $\Cl{n}$ that extends the identity automorphism $I$ of $\R^n$, called the {\it reversion}, denoted by $a\mapsto \tilde{a}$. Further, the {\it grade involution} denoted by $a\mapsto \hat{a}$ is a unique involutive automorphism of $\Cl{n}$ that extends $-I$, whereas the (Clifford) {\it conjugation} denoted by $a\mapsto \bar{a}$ is a unique involutive anti-automorphism of $\Cl{n}$ that extends $-I$. The grade involution is used to define a {\it Clifford group} (also called a {\it Lipschitz group}), which is a group under multiplication of $\Cl{n}$ defined by
\begin{equation}
\Gamma_n = \cset{g\in\Cl{n}}{g\textrm{ is invertible and }\hat{g}\vec{v}g^{-1}\in\R^n\textrm{ for all }\vec{v}\in\R^n}.
\end{equation}
The conjugation of $\Cl{n}$ gives rise to a group homomorphism of $\Gamma_n$. In fact, define a map $\eta$ by
\begin{equation}
\eta(a) = a\bar{a},\qquad a\in\Cl{n}.
\end{equation}
Then the restriction of $\eta$ to $\Gamma_n$ is a homomorphism from $\Gamma_n$ to the multiplicative group of nonzero numbers, denoted by $\mul{\R}$ \cite[Proposition 2]{TSKW2015EGB}. If an element $a$ in $\Cl{n}$ has the property that $\eta(a)\in\R$ and $\eta(a)\geq 0$, we define $\abs{a} = \sqrt{\eta(a)}$. It is not difficult to see that $\abs{\vec{v}} = \norm{\vec{v}}$ for all $\vec{v}\in\R^n$. 

The following theorem summarizes basic properties of $\Cl{n}$ that will be used in Section \ref{Sec: Main result}, especially the proof of Theorem \ref{thm: bound for Euclidean norm of Mobius addition}.

\begin{theorem}[Proposition 5, \cite{TSKW2015EGB}]\label{thm: Proposition 5 of paper EGB}
The following properties hold in the Clifford algebra $\Cl{n}$.
\begin{enumerate}
\item $\vec{u}\vec{v} + \vec{v}\vec{u} = -2\gen{\vec{u},\vec{v}}$ for all $\vec{u},
\vec{v}\in\R^n$.
\item\label{item: V^2 = -||v||^2} $\vec{v}^2 = -\norm{\vec{v}}^2$ for all $\vec{v}\in\R^n$.
\item $1-\vec{u}\vec{v}\in\Gamma_n$ and $(1-\vec{u}\vec{v})^{-1} = \dfrac{1-\vec{v}\vec{u}}{\eta(1-\vec{u}\vec{v})}$
for all $\vec{u}, \vec{v}\in\R^n$ with
$\norm{\vec{u}}\norm{\vec{v}}\ne 1$.
\item\label{item: squared modulus} $\eta(\vec{w}(1-\vec{u}\vec{v})^{-1}) =     \dfrac{\eta(\vec{w})}{\eta(1-\vec{u}\vec{v})}$
for all $\vec{u}, \vec{v}, \vec{w}\in\R^n$ with
$\norm{\vec{u}}\norm{\vec{v}}\ne 1$.
\end{enumerate}
\end{theorem}

In view of Theorem \ref{thm: Proposition 5 of paper EGB} \eqref{item: V^2 = -||v||^2}, if $\vec{v}\ne\vec{0}$, then $\vec{v}$ is invertible with respect to multiplication of $\Cl{n}$ and $\vec{v}^{-1} = -\dfrac{1}{\norm{\vec{v}}^2}\vec{v}$. Furthermore, by Lemma 1 of \cite{TSKW2015EGB}, 
$$
\hat{\vec{v}}\vec{w}\vec{v}^{-1} = \dfrac{1}{\norm{\vec{v}}^2}\vec{v}\vec{w}\vec{v}
$$
belongs to $\R^n$ for all nonzero vectors $\vec{v}\in\R^n$ and all $\vec{w}\in\R^n$. This implies that $\R^n\setminus\set{\vec{0}}\subseteq\Gamma_n$ and we obtain the following theorem.

\begin{theorem}\label{thm: invariant of modulus}
Every transformation of the form $\vec{w}\mapsto q\vec{w}q^{-1}$, where $\vec{w}\in\R^n$ and $q\in\Gamma_n$, defines an orthogonal transformation on $\R^n$.
\end{theorem}
\begin{proof}
Let $\vec{w}\in\R^n$ and let $q\in\Gamma_n$. Clearly, $\norm{q\vec{0}q^{-1}} = 0 = \norm{\vec{0}}$. Therefore, we may assume that $\vec{w}\ne \vec{0}$ and hence $\vec{w}\in\Gamma_n$. Since $\eta$ is a homomorphism from $\Gamma_n$ to $\mul{\R}$, it follows that $\eta(q\vec{w}q^{-1}) = \eta(q)\eta(\vec{w})\eta(q)^{-1} = \eta(\vec{w})$  and so 
\begin{equation*}
\norm{q\vec{w}q^{-1}} = \sqrt{\eta(q\vec{w}q^{-1})} = \sqrt{\eta(\vec{w})} = \norm{\vec{w}}.
\end{equation*}
It is clear that the map $\vec{w}\mapsto q\vec{w}q^{-1}$ is linear and bijective for $\vec{w}\mapsto q^{-1}\vec{w}q$ defines its inverse with respect to composition of maps.
\end{proof}

Using the Clifford algebra formalism, one gains a compact formula for M\"{o}bius addition, as shown in the following theorem.

\begin{theorem}[Theorem 5.2, \cite{JL2010CAM}]\label{thm: compact formula for Mobius addition}
In $\Cl{n}$, M\"{o}bius addition can be expressed as
\begin{equation}\label{Eqn: Simplified Mobius addition}
\vec{u}\oplus_M\vec{v} = (\vec{u}+\vec{v})(1-\vec{u}\vec{v})^{-1}
\end{equation}
for all $\vec{u},\vec{v}\in\B$. The gyroautomorphisms are given by $\gyr{\vec{u},\vec{v}}{\vec{w}} = q\vec{w}q^{-1}$, where 
$$q = \dfrac{1-\vec{u}\vec{v}}{\abs{1-\vec{u}\vec{v}}},$$
for all $\vec{u},\vec{v}, \vec{w}\in\B$.
\end{theorem}

\section{Metrics on the M\"{o}bius gyrogroup and their isometry groups}\label{Sec: Main result}
In this section, we prove a useful inequality involving M\"{o}bius addition and the Euclidean norm as an application of the Cauchy--Schwarz inequality, using the Clifford algebra formalism. This enables us to define a variant of norm metric on the M\"{o}bius gyrogroup. This metric turns out to be a characteristic property of M\"{o}bius transformations on $\hat{\R}^n$ carrying $\B$ onto itself, where $\hat{\R}^n$ is the one-point compactification of $\R^n$. We then give a complete description of the corresponding isometry group via a gyrogroup approach. 

\begin{theorem}\label{thm: bound for Euclidean norm of Mobius addition}
The inequality 
\begin{equation}\label{eqn: inequality in Monius ball}
\dfrac{\norm{\vec{u}}-\norm{\vec{v}}}{1+\norm{\vec{u}}\norm{\vec{v}}}\leq \norm{\vec{u}\oplus_M \vec{v}} \leq \dfrac{\norm{\vec{u}}+\norm{\vec{v}}}{1-\norm{\vec{u}}\norm{\vec{v}}}
\end{equation}
holds in the M\"{o}bius gyrogroup.
\end{theorem}
\begin{proof}
Using the Cauchy--Schwarz inequality, we have $$-\norm{\vec{u}}\norm{\vec{v}}\leq \gen{\vec{u},\vec{v}}\leq \norm{\vec{u}}\norm{\vec{v}}$$ for all $\vec{u},\vec{v}\in\R^n$.
This implies that
$$
\eta(\vec{u}+\vec{v}) = \norm{\vec{u}}^2 -(\vec{u}\vec{v}+\vec{v}\vec{u})+\norm{\vec{v}}^2 = \norm{\vec{u}}^2 +2\gen{\vec{u},\vec{v}}+\norm{\vec{v}}^2 \leq  (\norm{\vec{u}}+\norm{\vec{v}})^2
$$
and that $\eta(\vec{u}+\vec{v}) \geq (\norm{\vec{u}}-\norm{\vec{v}})^2$ for all $\vec{u},\vec{v}\in\R^n$. Let $\vec{u},\vec{v}\in\B$. As in the proof of Proposition 5 (4) of \cite{TSKW2015EGB}, we have $\eta(1-\vec{u}\vec{v})\geq (1-\norm{\vec{u}}\norm{\vec{v}})^2$ and 
$$\eta(1-\vec{u}\vec{v})  = 1 + 2\gen{\vec{u},\vec{v}}+ \norm{\vec{u}}^2\norm{\vec{v}}^2 \leq  (1+\norm{\vec{u}}\norm{\vec{v}})^2.$$ Hence, by Theorem \ref{thm: Proposition 5 of paper EGB} \eqref{item: squared modulus}, 
$$
\norm{\vec{u}\oplus_M\vec{v}} = \sqrt{\dfrac{\eta(\vec{u}+\vec{v})}{\eta(1-\vec{u}\vec{v})}}\leq \sqrt{\dfrac{(\norm{\vec{u}}+\norm{\vec{v}})^2}{(1-\norm{\vec{u}}\norm{\vec{v}})^2}} = \dfrac{\norm{\vec{u}}+\norm{\vec{v}}}{1-\norm{\vec{u}}\norm{\vec{v}}}
$$
and similarly
$$
\norm{\vec{u}\oplus_M\vec{v}} = \sqrt{\dfrac{\eta(\vec{u}+\vec{v})}{\eta(1-\vec{u}\vec{v})}}\geq \sqrt{\dfrac{(\norm{\vec{u}}-\norm{\vec{v}})^2}{(1+\norm{\vec{u}}\norm{\vec{v}})^2}} \geq \dfrac{\norm{\vec{u}}-\norm{\vec{v}}}{1+\norm{\vec{u}\norm{\vec{v}}}},
$$
as required.
\end{proof}

In view of \eqref{eqn: inequality in Monius ball} and the well known trigonometric identity, the tangent function is needed in order to obtain a bounded metric on the unit ball of $\R^n$. In fact, define a function $\norm{\cdot}_T$ by
\begin{equation}
\norm{\vec{v}}_T = \tan^{-1}{\norm{\vec{v}}}
\end{equation}
for all $\vec{v}\in\B$. Here, $T$ stands for \qt{$\tan^{-1}$}.

\begin{theorem}\label{thm: gyronorm on B induced by tan}
$\norm{\cdot}_T$ satisfies the following properties:
\begin{enumerate}
\item\label{item: positive, tan} $\norm{\vec{x}}_T\geq 0$ and $\norm{\vec{x}}_T = 0$ if and only if $\vec{x} = \vec{0}$;
\item\label{item: invariant inverse, tan} $\norm{\ominus \vec{x}}_T = \norm{\vec{x}}_T$;
\item\label{item: subadditive, tan} $\norm{\vec{x}}_T - \norm{\vec{y}}_T \leq \norm{\vec{x}\oplus_M\vec{y}}_T \leq \norm{\vec{x}}_T + \norm{\vec{y}}_T$;
\item\label{item: invariant gyration, tan} $\norm{\gyr{\vec{u},\vec{v}}{\vec{x}}}_T = \norm{\vec{x}}_T$
\end{enumerate}
for all $\vec{u}, \vec{v}, \vec{x}, \vec{y}\in\B$.
\end{theorem}
\begin{proof}
Item \eqref{item: positive, tan} follows from the fact that $\tan^{-1}$ is a strictly increasing injective function on $(-\infty, \infty)$. Item \eqref{item: invariant inverse, tan} follows from the fact that $\norm{-\vec{x}} = \norm{\vec{x}}$.

To prove \eqref{item: subadditive, tan}, set $x = \tan^{-1}{\norm{\vec{x}}}$ and $y = \tan^{-1}{\norm{\vec{y}}}$. By Theorem \ref{thm: bound for Euclidean norm of Mobius addition},
$$
\dfrac{\norm{\vec{x}}-\norm{\vec{y}}}{1+\norm{\vec{x}}\norm{\vec{y}}}\leq \norm{\vec{x}\oplus_M \vec{y}} \leq \dfrac{\norm{\vec{x}}+\norm{\vec{y}}}{1-\norm{\vec{x}}\norm{\vec{y}}}
$$
and so $\tan{(x-y)} \leq \norm{\vec{x}\oplus_M \vec{y}} \leq \tan{(x+y)}$. Since $\tan^{-1}$ is an increasing function, it follows that 
$
x - y  \leq \tan^{-1}{\norm{\vec{x}\oplus_M \vec{y}}}\leq x + y
$, as claimed. By Theorem \ref{thm: compact formula for Mobius addition}, there is an element $q\in\Gamma_n$ for which $\gyr{\vec{u},\vec{v}}{\vec{x}} = q\vec{x}q^{-1}$. It follows from Theorem \ref{thm: invariant of modulus} that
$$
\norm{\gyr{\vec{u},\vec{v}}{\vec{x}}}_T = \tan^{-1}{\norm{q\vec{x}q^{-1}}} = \tan^{-1}{\norm{\vec{x}}} = \norm{\vec{x}}_T,
$$
which proves \eqref{item: invariant gyration, tan}.
\end{proof}

As a consequence of Theorem \ref{thm: gyronorm on B induced by tan}, we obtain a new metric on the M\"{o}bius gyrogroup. Unlike the Poincar\'{e} metric, this metric is bounded as shown in the following theorem.

\begin{theorem}
Define $d_T$ by
\begin{equation}
d_T(\vec{x},\vec{y}) = \norm{\ominus \vec{x}\oplus_M\vec{y}}_T
\end{equation}
for all $\vec{x}, \vec{y}\in\B$. Then $d_T$ is a bounded metric on $\B$.
\end{theorem}
\begin{proof}
By Theorem \ref{thm: gyronorm on B induced by tan} \eqref{item: positive, tan}, $d_T(\vec{x},\vec{y})\geq 0$ for all $\vec{x},\vec{y}\in\B$ and $d_T(\vec{x},\vec{y}) = 0$ if and only if $\vec{x} = \vec{y}$. Let $\vec{x}, \vec{y}, \vec{z}\in \B$. Using appropriate properties of the M\"{o}bius gyrogroup in Table \ref{tab: properties of gyrogroups}, together with Theorem \ref{thm: gyronorm on B induced by tan}, we obtain 
$$
\norm{\ominus \vec{y}\oplus_M \vec{x}}_T =
 \norm{\ominus(\ominus \vec{y}\oplus_M \vec{x})}_T
= \norm{\gyr{\ominus \vec{y}, \vec{x}}{(\ominus \vec{x}\oplus_M \vec{y})}}_T
= \norm{\ominus \vec{x}\oplus_M \vec{y}}_T
$$
and so $d_T(\vec{y}, \vec{x}) = d_T(\vec{x}, \vec{y})$. Furthermore, we obtain
\begin{align*}
d_T(\vec{x}, \vec{z}) &= \norm{\ominus \vec{x}\oplus_M \vec{z}}_T\\
{} &= \norm{(\ominus \vec{x}\oplus_M \vec{y})\oplus_M \gyr{\ominus \vec{x}, \vec{y}}{(\ominus \vec{y}\oplus_M \vec{z})}}_T\\
{} &\leq \norm{\ominus \vec{x}\oplus_M \vec{y}}_T+\norm{\gyr{\ominus \vec{x}, \vec{y}}{(\ominus \vec{y}\oplus_M \vec{z})}}_T\\
{} &= \norm{\ominus \vec{x}\oplus_M \vec{y}}_T+ \norm{\ominus \vec{y}\oplus_M \vec{z}}_T\\
{} &= d_T(\vec{x}, \vec{y}) + d_T(\vec{y}, \vec{z}).
\end{align*}
This proves that $d_T$ satisfies the defining properties of a metric.

Note that $d_T(\vec{0}, \vec{v}) = \norm{\vec{v}}_T= \tan^{-1}{\norm{\vec{v}}} < \tan^{-1}{1} = \dfrac{\pi}{4}$ 
for all $\vec{v}\in\B$. Hence,
$$
d_T(\vec{x}, \vec{y}) \leq d_T(\vec{x},\vec{0})+d_T(\vec{0},\vec{y}) < \dfrac{\pi}{4}+\dfrac{\pi}{4} = \dfrac{\pi}{2}
$$
for all $\vec{x},\vec{y}\in\B$.
\end{proof}

Although $d_T$ is quite different from the Poincar\'{e} metric, both generate the same topology on the unit ball. It is clear that the Poincar\'{e} metric and the rapidity metric of the M\"{o}bius gyrogroup generate the same topology since the former is twice the latter.

\begin{theorem}
The topologies induced by $d_T$ and $d_M$ are equivalent.
\end{theorem}
\begin{proof}
Note that $d_T(\vec{u}, \vec{v}) \leq d_M(\vec{u}, \vec{v})$
for all $\vec{u}, \vec{v}\in\B$ since $$f(x) = \tanh^{-1}{x} - \tan^{-1}{x}$$ defines a strictly increasing function on the open interval $(0, 1)$. This implies that the topology generated by $d_M$ is finer than the topology generated by $d_T$. Next, we prove that the topology generated by $d_T$ is finer than the topology generated by $d_M$. Let $\vec{u}\in\B$ and let $\epsilon>0$. Choose $\delta = \tan^{-1}{(\tanh{\epsilon})}$. Let $\vec{v}\in B_{d_T}(\vec{u}, \delta)$. Then $d_T(\vec{u}, \vec{v})<\delta$, that is,
$\norm{\ominus \vec{u}\oplus_M\vec{v}}_T < \tan^{-1}{(\tanh{\epsilon})}$. It follows that
$$
d_M(\vec{u},\vec{v}) = \tanh^{-1}{\norm{\ominus \vec{u}\oplus_M\vec{v}}} < \epsilon
$$
for $\tan$ and $\tanh^{-1}$ are strictly increasing functions. Hence, $\vec{v}\in B_{d_M}(\vec{u}, \epsilon)$. This proves $B_{d_T}(\vec{u}, \delta)\subseteq B_{d_M}(\vec{u}, \epsilon)$. 
\end{proof}

Let $\Or{\R^n}$ be the orthogonal group of $\R^n$, that is,
\begin{equation}
\Or{\R^n} = \cset{\tau}{\tau\textrm{ is a bijective orthogonal transformation on }\R^n}.
\end{equation}
Set
\begin{equation}
\Or{\B} = \cset{\res{\tau}{\B}}{\tau\in\Or{\R^n}},
\end{equation}
where $\res{\tau}{\B}$ is the restriction of $\tau$ to $\B$. It is clear that $\Or{\B}$ forms a group under composition of maps since $\B$ is preserved under orthogonal transformations on $\R^n$. Given $\vec{u},\vec{v}\in \B$, note that $\gyr{\vec{u}, \vec{v}}{}$ satisfies the following properties:
\begin{enumerate}
\item $\gyr{\vec{u}, \vec{v}}{\vec{0}} = \vec{0}$;
\item $\gyr{\vec{u}, \vec{v}}{}$ is an automorphism of $(\B,\oplus_M)$;
\item $\gyr{\vec{u}, \vec{v}}{}$ preserves the M\"{o}bius gyrometric.
\end{enumerate}
Hence, by Theorem 3.2 of \cite{TA2014GPM}, there is a bijective orthogonal transformation on $\R^n$, denoted by $\Gyr{\vec{u}, \vec{v}}{}$, for which $\res{\Gyr{\vec{u}, \vec{v}}{}}{\B} = \gyr{\vec{u}, \vec{v}}{}$. This proves the following inclusion:
$$
\cset{\gyr{\vec{u}, \vec{v}}{}}{\vec{u},\vec{v}\in\B}\subseteq \Or{\B}.
$$
Next, we compute the isometry group of $(\B, d_T)$.

\begin{lemma}\label{lem: tangent left gyrotranslation isometry}
The left  gyrotranslation $L_{\vec{u}}\colon \vec{v}\mapsto \vec{u}\oplus_M\vec{v}$ defines an isometry of $(\B, d_T)$ for all $\vec{u}\in\B$.
\end{lemma}
\begin{proof}
By Theorem 10 (1) of \cite{TSKW2015ITG}, $L_{\vec{u}}$ is a bijective self-map of $\B$. Using \mbox{appropriate} properties of the M\"{o}bius gyrogroup in Table \ref{tab: properties of gyrogroups}, we obtain 
\begin{align*}
\norm{\ominus (\vec{u}\oplus_M \vec{x})\oplus_M (\vec{u}\oplus_M \vec{y})} &= \norm{\gyr{\vec{u}, \vec{x}}{(\ominus \vec{x}\ominus \vec{u})}\oplus_M(\vec{u}\oplus_M \vec{y})}\\
{} &= \norm{(\ominus \vec{x}\ominus \vec{u})\oplus_M\gyr{\vec{x}, \vec{u}}{(\vec{u}\oplus_M \vec{y})}}\\
{} &= \norm{(\ominus \vec{x}\ominus \vec{u})\oplus_M\gyr{\ominus \vec{x}, \ominus \vec{u}}{(\vec{u}\oplus_M \vec{y})}}\\
{} &= \norm{\ominus \vec{x}\oplus_M \vec{y}}.
\end{align*}
It follows that 
\begin{equation*}
d_T(L_\vec{u}(\vec{x}), L_\vec{u}(\vec{y})) = \norm{\ominus L_{\vec{u}}(\vec{x})\oplus_M L_{\vec{u}}(\vec{y})}_T = \norm{\ominus \vec{x}\oplus_M \vec{y}}_T = d_T(\vec{x}, \vec{y}). \qedhere
\end{equation*}
\end{proof}

\begin{theorem}\label{thm: Isometry group of (B, dT)}
The isometry group of $(\B, d_T)$ is given by
\begin{equation}
\Iso{\B, d_T} = \cset{L_\vec{u}\circ \tau}{\vec{u}\in\B, \tau\in\Or{\B}}.
\end{equation}
\end{theorem}
\begin{proof}
For convenience, if $\rho\in\Or{\R^n}$, then the restriction of $\rho$ to $\B$ is simply denoted by $\rho$. By Lemma \ref{lem: tangent left gyrotranslation isometry}, $L_\vec{u}$ is an isometry of $\B$ with respect to $d_T$. Let $\rho\in\Or{\R^n}$. Using \eqref{eqn: Euclidean Mobius addition}, we have $\rho(\vec{x})\oplus_M \rho(\vec{y}) = \rho(\vec{x}\oplus_M\vec{y})$ for all $\vec{x},\vec{y}\in\B$ since $\rho$ is linear and preserves the Euclidean inner product. Hence, the restriction of $\rho$ to $\B$ is indeed an automorphism of $(\B, \oplus_M)$ since $\rho(\B)\subseteq\B$ and $\rho^{-1}\in\Or{\R^n}$. It follows that
$$
d_T(\rho(\vec{x}), \rho(\vec{y})) = \norm{\rho(\ominus \vec{x}\oplus_M \vec{y})}_T = \norm{\ominus \vec{x}\oplus_M \vec{y}}_T = d_T(\vec{x}, \vec{y}).
$$
Thus, $\rho$ is an isometry of $\B$ and so $\cset{L_\vec{u}\circ \tau}{\vec{u}\in\B, \tau\in\Or{\B}}\subseteq \Iso{\B, d_T}$.

Let $T\in \Iso{\B, d_T}$. By definition, $T$ is a bijective self-map of $\B$. By Theorem 11 of \cite{TSKW2015ITG}, $T = L_{T(\vec{0})}\circ\rho$, where $\rho$ is a bijective self-map of $\B$ fixing $\vec{0}$. As in the proof of Theorem 18 (2) of \cite{TS2016TAG}, $ L_{T(\vec{0})}^{-1} = L_{\ominus T(\vec{0})}$ and so $\rho = L_{\ominus T(\vec{0})}\circ T$. Therefore, $\rho$ is an isometry of $(\B, d_T)$. Since $d_T(\rho(\vec{x}), \rho(\vec{y})) = d_T(\vec{x}, \vec{y})$ and $\tan^{-1}$ is injective, it follows that
$$
\norm{\ominus \rho(\vec{x})\oplus_M \rho(\vec{y})} = \norm{\ominus \vec{x}\oplus_M\vec{y}}
$$
for all $\vec{x},\vec{y}\in\B$. Thus, $\rho$ preserves the M\"{o}bius gyrometric. By Theorem 3.2 of \cite{TA2014GPM}, $\rho = \res{\tau}{\B}$, where $\tau$ is a bijective orthogonal transformation on $\R^n$. This proves that
\begin{equation*}
\Iso{\B, d_T} \subseteq \cset{L_\vec{u}\circ \tau}{\vec{u}\in\B, \tau\in\Or{\B}}.\qedhere
\end{equation*}
\end{proof}        

By Theorem \ref{thm: Isometry group of (B, dT)}, every isometry of $\B$ with respect to $d_T$ can be expressed as the composite of a left gyrotranslation with an orthogonal transformation restricted to $\B$. This expression is unique in the sense that if $L_\vec{u}\circ \alpha = L_{\vec{v}}\circ \beta$ with $\vec{u},\vec{v}$ in $\B$ and $\alpha,\beta$ in $\Or{\B}$, then $\vec{u} = \vec{v}$ and $\alpha = \beta$. Furthermore, we have the following composition law of isometries of $(\B, d_T)$:
\begin{equation}
(L_\vec{u}\circ \alpha)\circ(L_{\vec{v}}\circ \beta) = L_{\vec{u}\oplus_M \alpha(\vec{v})}\circ (\gyr{\vec{u}, \alpha(\vec{v})}{}\circ\alpha\circ\beta)
\end{equation}
for all $\vec{u},\vec{v}\in\B$ and $\alpha,\beta\in\Or{\B}$, a formula comparable to the composition law of Euclidean isometries.

Since $\vec{v}\mapsto L_{\vec{v}}$ defines a one-to-one correspondence from $\B$ to the set of left gyrotranslations of $\B$, we have
\begin{equation}
\Iso{\B, d_T}\cong \B\rtimes_{\rm gyr} \Or{\B}.
\end{equation}
Here, $\B\rtimes_{\rm gyr} \Or{\B}$ is the semidirect-product-like group whose underlying set is
\begin{equation}
\B\rtimes_{\rm gyr} \Or{\B} = \cset{(\vec{v}, \tau)}{\vec{v}\in\B, \tau\in\Or{\B}}
\end{equation}
with group law
\begin{equation}\label{eqn: group law of Bxgyr O(B)}
(\vec{u}, \alpha)(\vec{v}, \beta) = (\vec{u}\oplus_M \alpha(\vec{v}), \gyr{\vec{u}, \alpha(\vec{v})}{}\circ\alpha\circ\beta)
\end{equation}
for all $\vec{u},\vec{v}\in\B$ and $\alpha,\beta\in\Or{\B}$. This is a result analogous to the fact that the isometry group of the Euclidean space is the semidirect product of $\R^n$ and $\Or{\R^n}$:
$$
\R^n\rtimes \Or{\R^n} = \cset{(\vec{v}, \tau)}{\vec{v}\in\R^n, \tau\in\Or{\R^n}},
$$
where the group law is given by
$$
(\vec{u}, \alpha)(\vec{v}, \beta) = (\vec{u} + \alpha(\vec{v}), \alpha\circ\beta)
$$
for all $\vec{u},\vec{v}\in\R^n$ and $\alpha,\beta\in\Or{\R^n}$. The group $\B\rtimes_{\rm gyr} \Or{\B}$ is known as the \mbox{\it gyrosemidirect} of $\B$ and $\Or{\B}$ \cite[Section 2.6]{AU2008AHG}.

\begin{theorem}
Let $T$ be a self-map of $\B$. The following are equivalent:
\begin{enumerate}
\item $T$ preserves the Poincar\'{e} metric $d_P$;
\item $T$ preserves the rapidity metric $d_M$;
\item $T$ preserves the M\"{o}bius gyrometric $\varrho_M$;
\item $T$ preserves the metric $d_T$ generated by $\norm{\cdot}_T$.
\end{enumerate}
\end{theorem}
\begin{proof}
The theorem follows directly from the fact that $d_P(\vec{x}, \vec{y}) = 2d_M(\vec{x}, \vec{y})$ and that $\tanh^{-1}$ and $\tan^{-1}$ are injective.
\end{proof}

\begin{corollary}\label{cor: Isometry groups}
$\Iso{\B, d_P} = \Iso{\B, d_M} = \Iso{\B, \varrho_M} = \Iso{\B, d_T}$.
\end{corollary}

Recall that a M\"{o}bius transformation of $\hat{\R}^n$ that leaves $\B$ invariant  is called a M\"{o}bius transformation of $\B$ \cite[p. 120]{JR2006FHM}. It is known that the isometry group of the Poincar\'{e} ball model $(\B, d_P)$, also called the {\it conformal ball model}, can be identified with the group of M\"{o}bius transformations of $\B$; see, for instance, \cite[Corollary 1 on p. 125]{JR2006FHM}. By Corollary \ref{cor: Isometry groups}, Equation \eqref{eqn: group law of Bxgyr O(B)} provides a parametric realization of the M\"{o}bius transformation group of $\B$ in terms of vectors and rotations. Further, $d_T$ is an invariant of M\"{o}bius transformations of $\B$ in the sense of the following theorem. 

\begin{theorem}
Every M\"{o}bius transformation of $\B$ restricts to an isometry of $(\B, d_T)$, and every isometry of $(\B, d_T)$ extends to a unique M\"{o}bius transformation of $\B$.
\end{theorem}
\begin{proof}
Let $\phi$ be a M\"{o}bius transformation of $\B$. By Theorem 4.5.2 of \cite{JR2006FHM}, $\phi$ restricts to an isometry of $(\B, d_P)$. By Corollary \ref{cor: Isometry groups}, $\res{\phi}{\B}$ is an isometry of $(\B, d_T)$. Let $\sigma$ be an isometry of $(\B, d_T)$. By the same corollary, $\sigma$ is an isometry of $(\B, d_P)$ and hence extends to a unique M\"{o}bius transformation of $\B$ by the same theorem.
\end{proof}


\bibliographystyle{amsplain}\addcontentsline{toc}{section}{References}
\bibliography{References}

\providecommand{\bysame}{\leavevmode\hbox to3em{\hrulefill}\thinspace}
\providecommand{\MR}{\relax\ifhmode\unskip\space\fi MR }
\providecommand{\MRhref}[2]{%
  \href{http://www.ams.org/mathscinet-getitem?mr=#1}{#2}
}
\providecommand{\href}[2]{#2}
\begin{thebibliography}{10}

\bibitem{TA2014GPM}
T.~Abe, \emph{Gyrometric preserving maps on {E}instein gyrogroups, {M}\"{o}bius
  \mbox{gyrogroups} and {P}roper {V}elocity gyrogroups}, Nonlinear Funct. Anal.
  Appl. \textbf{19} (2014), 1--17.

\bibitem{MFGR2011MGC}
M.~Ferreira and G.~Ren, \emph{{M}\"{o}bius gyrogroups: A {C}lifford algebra
  approach}, J. Algebra \textbf{328} (2011), 230--253.

\bibitem{YFTS2005PAH}
Y.~Friedman and T.~Scarr, \emph{Physical applications of homogeneous balls},
  Progress in Mathematical Physics, vol.~40, Birkh\"{a}user, Boston, 2005.

\bibitem{SKJL2013UBL}
S.~Kim and J.~Lawson, \emph{Unit balls, {L}orentz boosts, and hyperbolic
  geometry}, \mbox{Results} Math. \textbf{63} (2013), 1225--1242.

\bibitem{JL2010CAM}
J.~Lawson, \emph{{C}lifford algebras, {M}\"{o}bius transformations, {V}ahlen
  matrices, and {B}-loops}, Comment. Math. Univ. Carolin. \textbf{51} (2010),
  no.~2, 319--331.

\bibitem{JR2006FHM}
J.~Ratcliffe, \emph{Foundations of hyperbolic manifolds}, 2nd ed., Graduate
  Texts in Mathematics, vol. 149, Springer, New York, 2006.

\bibitem{TS2016TAG}
T.~Suksumran, \emph{Essays in mathematics and its applications: In honor of
  {V}ladimir {A}rnold}, ch.~The Algebra of Gyrogroups: {C}ayley's Theorem,
  \mbox{{L}agrange}'s Theorem, and Isomorphism Theorems, pp.~369--437,
  Springer, Switzerland, 2016.

\bibitem{TSKW2015EGB}
T.~Suksumran and K.~Wiboonton, \emph{{E}instein gyrogroup as a {B}-loop}, Rep.
  Math. Phys. \textbf{76} (2015), 63--74.

\bibitem{TSKW2015ITG}
\bysame, \emph{Isomorphism theorems for gyrogroups and
  \mbox{{L}-subgyrogroups}}, J. Geom. Symmetry Phys. \textbf{37} (2015),
  67--83.

\bibitem{AU2008AHG}
A.~Ungar, \emph{Analytic hyperbolic geometry and {A}lbert {E}instein's
  {S}pecial \mbox{{T}heory} {o}f {R}elativity}, World Scientific, Hackensack,
  NJ, 2008.

\bibitem{AU2008FMG}
\bysame, \emph{From {M}\"{o}bius to gyrogroups}, Amer. Math. Monthly
  \textbf{115} (2008), no.~2, 138--144.

\end{thebibliography}
\end{document}